\documentclass[12pt,reqno]{article}

\usepackage[usenames,dvipsnames]{xcolor}
\usepackage{graphicx}
\usepackage{amscd}
\usepackage{mathdots}
 
\usepackage{comment}
 
\usepackage[colorlinks=true,
linkcolor=webgreen,
filecolor=webbrown,
citecolor=webgreen]{hyperref}

\definecolor{webgreen}{rgb}{0,.5,0}
\definecolor{webbrown}{rgb}{.6,0,0}

\usepackage{color}
\usepackage{fullpage}
\usepackage{float}

\usepackage{graphics,amsmath,amssymb}
\usepackage{amsthm}
\usepackage{amsfonts}
\usepackage{latexsym}
\usepackage{epsf}
\usepackage{asymptote}
\usepackage{graphicx}

\usepackage{todonotes}

\theoremstyle{plain}
\newtheorem{theorem}{Theorem}
\newtheorem{lemma}[theorem]{Lemma}
\newtheorem{corollary}[theorem]{Corollary}

\theoremstyle{definition}
\newtheorem{definition}[theorem]{Definition}
\newtheorem{example}[theorem]{Example}

\newtheorem{remark}[theorem]{Remark}

\title{Enumerating Anchored Permutations \\ with Bounded Gaps}
\author{
Maria M.\ Gillespie \thanks{email: \texttt{maria.gillespie@colostate.edu}; Corresponding author; Supported by NSF grant PDRF 1604262.} \\
Colorado State University \\
Fort Collins, CO 80523
\and 
Kenneth G.\ Monks \thanks{email: \texttt{monks@scranton.edu}} \\
University of Scranton \\ Scranton, PA 18510
\and
Kenneth M.\ Monks \thanks{email: \texttt{kenneth.monks@frontrange.edu}} \\
Front Range Community College \\ Longmont, CO 80501
}

\date{\today}

\begin{document}

\renewcommand{\thefootnote}{\fnsymbol{footnote}}
\begin{NoHyper}\footnotetext{\emph{Keywords}: Permutation patterns, rational generating functions, algebraic graph theory, spectral graph theory.}\end{NoHyper}
\renewcommand{\thefootnote}{\arabic{footnote}}

\maketitle

\vskip .2 in

\baselineskip14pt

\begin{abstract}
  Say that a permutation of $1,2,\ldots,n$ is \textit{$k$-bounded} if every pair of consecutive entries in the permutation differs by no more than $k$.  Such a permutation is \textit{anchored} if the first entry is $1$ and the last entry is $n$.   We show that the generating function for the enumeration of $k$-bounded anchored permutations is always rational, mirroring the known result on (non-anchored) $k$-bounded permutations due to Avgustinovich and Kitaev.  We then explicitly determine the recursive formulas of minimal depth for the number of anchored $k$-bounded permutations of $n$ for $k=2$ and $k=3$, resolving a conjecture listed on the Online Encyclopedia of Integer Sequences (entry A249665).
  We additionally show that the number of anchored $k$-bounded permutations of $n$ is asymptotically $O\left(k^n\right)$ as a function of $n$ for a given $k$.
\end{abstract}

\section{Introduction}

Suppose one starts on the first stair of a staircase with $n$ steps labeled $1,\ldots,n$ in order, and at each step one either steps forwards or backwards by at most $k$ steps, such that every stair is used exactly once and the climb ends on the $n$th stair.  How many distinct such ways are there to climb the stairs?

This question can be stated more precisely as follows.  For a positive integer $k$, define a \textbf{$k$-bounded} permutation of $[n]=\{1,2,\ldots,n\}$ to be a bijection $\pi:[n]\to [n]$ such that for all $i\in \{1,2,\ldots,n-1\}$ we have $$\left|\pi(i)-\pi(i+1)\right|\leq k.$$  We say that a permutation is \textbf{anchored} if $\pi(1)=1$ and $\pi(n)=n$.  We are interested in enumerating the $k$-bounded anchored permutations in terms of $k$ and $n$.

\begin{example}
  The permutation $1,4,2,3,6,5,7,8,9$ is a $3$-bounded anchored permutation of $\{1,2,\ldots,9\}$, since the first entry is $1$, the last entry is $9$, and no pair of consecutive entries differs by more than $3$.
\end{example}

Several related questions have been previously explored.  Positive stair climbing problems were studied by Goins and Washington \cite{Standard-stair-climbing}, extending the well-known fact that the number of ways to climb a staircase of length $n$ using positive steps of $+1$ or $+2$ each time is the $n$th Fibonacci number.  

Avgustinovich and Kitaev \cite{AK} studied \emph{$k$-determined permutations}, which they show are equivalent to (the inverses of) $(k-1)$-bounded, \textit{non-anchored} permutations, as well as certain Hamiltonian paths in graphs.  They resolve a conjecture of Plouffe \cite{Plouffe} by providing the generating function for $2$-bounded non-anchored permutations, which were originally defined as \textit{key permutations} \cite{Page}.  Avgustinovich and Kitaev further show that the generating function of the $k$-determined permutations of $n$ is always rational for any $k$, using the transfer-matrix method described by Stanley \cite[ch.\ 4]{Stanley}.  

In this paper, we show that the generating function is still rational for \textit{anchored} $k$-bounded permutations, for any $k$. Furthermore, we resolve  the question of finding a minimal-depth linear recurrence for the enumeration of $3$-bounded anchored permutations, which was originally posed as a conjecture on the Online Encyclopedia of Integer Sequences, entry A249665 \cite{OEIS}.

\subsection{Main results}

Let $F_n^{(k)}$ be the number of $k$-bounded anchored permutations of $[n]$.  For $k=1$, there is clearly only one $1$-bounded anchored permutation for each $n$, namely the identity permutation.  In this paper, we resolve the cases $k=2$ and $k=3$ completely by finding recursions for $F^{(2)}_n$ and $F^{(3)}_n$.  We also adapt the methods in \cite{AK} to show that the generating function $$F^{(k)}(x)=\sum F_n^{(k)}x^n$$ of $k$-bounded anchored permutations is rational for any $k$, which in particular implies that there always exists a finite-depth homogeneous linear recurrence relation for enumerating these permutations.  We then apply techniques from algebraic graph theory to obtain an asymptotic upper bound for $F_n^{(k)}$.

Our main results are summarized in the following theorems. 

\begin{theorem}\label{thm:rational}
    The generating function $F^{(k)}(x)$ is a rational function, that is, $F^{(k)}(x)\in \mathbb{Z}(x)$.
\end{theorem}

Due to standard results on rational generating functions (see \cite[ch.\ 4]{Stanley}), we immediately obtain the following corollary.

\begin{corollary}
  For any fixed $k$, the numbers $F_n^{(k)}$ satisfy a finite-depth homogeneous linear recurrence in $n$ with integer coefficients.
\end{corollary}

The proof of Theorem 4 relies on the transfer-matrix method, giving an expression for the rational function that is explicit, but may not be a reduced fraction.  It therefore gives an explicit recursion for $F_n^{(k)}$ for all $k$, but this recursion may not be the minimal depth recursion.  We give the reduced generating function and minimal recursion for $k=2$ and $k=3$ in the next two theorems.

For simplicity, we define $R_n=F_n^{(2)}$ and $F_n=F_n^{(3)}$ in the theorems below. 

\begin{theorem}\label{thm:2-bounded}
  Let $R_n$ be the number of $2$-bounded anchored permutations of $[n]$.  Then the sequence $\left(R_n\right)_{n\geq1}$ is given by the recurrence $R_1=1$, $R_2=1$, $R_3=1$, and 
    \begin{equation}\label{eqn:recursion2}
      R_n=R_{n-1}+R_{n-3}
    \end{equation} 
  for all $n\geq 4$.  The generating function of the sequence is   
    $$R(x)=\sum_{n=1}^\infty R_n x^n=\frac{x}{1-x-x^3}.$$   
\end{theorem}

This sequence $R_n$ is also known as \textit{Narayana's cows sequence} \cite{OEIS-Narayana-Cows}, and the particular interpretation as $2$-bounded anchored permutation is stated without proof (in a slightly different but equivalent form) in Flajolet and Sedgewick \cite[p.\ 373]{Flajolet}. Note the similarity to the Fibonacci recurrence.  It is interesting that for steps of $+1$ and $+2$ only, the recurrence is precisely the Fibonacci sequence, and here, with the added steps of $-1$ and $-2$ where every step is reached, it is one index off of the Fibonacci recurrence.

  We provide two separate proofs of Theorem \ref{thm:2-bounded}.  The first uses the transfer-matrix method used to prove Theorem \ref{thm:rational}, which we include in order to illustrate the method.  The second is a direct, elegant combinatorial proof. 

\begin{theorem}\label{thm:3-bounded}
  Let $F_n$ be the number of $3$-bounded anchored permutations of $[n]$.  Then the sequence $\left(F_n\right)_{n\geq1}$ is given by the recurrence $F_1=1$, $F_2=1$, $F_3=1$, $F_4=2$, $F_5=6$, $F_6=14$, $F_7=28$, $F_8=56$, and 
    \begin{equation}\label{eqn:recursion3}
      F_n=2F_{n-1}-F_{n-2}+2F_{n-3}+F_{n-4}+F_{n-5}-F_{n-7}-F_{n-8}
    \end{equation}
  for all $n\geq 9$.  The generating function of the sequence is
    $$F(x)=\frac{x-x^2-x^4}{1-2x+x^2-2x^3-x^4-x^5+x^7+x^8}.$$ 
\end{theorem}

Lastly, we determine an upper bound for the Perron-Frobenius eigenvalue of the transfer matrix mentioned above.  We use this to obtain an asymptotic bound for $F_n^{(k)}$.

\begin{theorem}\label{thm:asymptotics}
  For any $k$, we have that $F_n^{(k)}$ is $O\left(k^n\right)$ as $n\to \infty$.
\end{theorem}

We prove Theorems  \ref{thm:rational}, \ref{thm:2-bounded}, \ref{thm:3-bounded}, and \ref{thm:asymptotics} in Sections \ref{sec:rational}, \ref{sec:2-bounded},  \ref{sec:3-bounded}, and \ref{sec:asymptotics} respectively. 

\section{Rationality of the generating function for all \texorpdfstring{$k$}{k}}\label{sec:rational}

We now prove Theorem \ref{thm:rational}.  To do so, we first recall the basic well-known notions of permutation patterns, as well as important definitions and facts from \cite{AK}.  Throughout this section we write $S_m$ for the set of all permutations of $\{1,2,\ldots,m\}$.  We also fix positive integers $k$ and $n$ with $k\le n$ throughout.

\begin{definition}
  A \textbf{consecutive pattern} of length $k$ of a permutation $\pi\in S_n$ is a permutation $\tau\in S_k$ for which the relative order of the entries $\tau(1),\ldots,\tau(k)$ exactly matches that of $k$ consecutive entries $\pi(i+1),\pi(i+2)\cdots,\pi(i+k)$ of $\pi$ for some $i\in\{0,1,\ldots,n-k\}$.  
\end{definition}

For instance, the consecutive patterns of length $3$ of the permutation $51432$ are $312$, $132$, and $321$ (from left to right).  We can think of these consecutive patterns as defining a path in the following graph defined in \cite{AK}, which is similar to the well-known \textit{DeBruijn graphs}.

\begin{definition}[\cite{AK}]
   The \textbf{graph of pattern overlaps of length $k$} is the directed graph $\mathcal{P}_k$ whose vertex set is $S_k$ and for which there is a directed edge from $\tau$ to $\sigma$ if and only if the pattern of $\tau(2),\ldots,\tau(k)$ equals the pattern of $\sigma(1),\ldots,\sigma(k-1)$.
\end{definition}

If $k\le n$, then any permutation $\pi$ of $\{1,2,\ldots,n\}$ corresponds to a unique path in $\mathcal{P}_k$, given by considering the sequence of consecutive patterns of length $k$ in $\pi$ from left to right.  For example, the path corresponding to $51432$ in $\mathcal{P}_3$ is $$312\longrightarrow 132 \longrightarrow 321.$$ However, not all paths determine a unique permutation, as there may be more than one permutation corresponding to the same path. For instance, $52431$ has the same path as the example $51432$ above. 
This naturally leads one to consider the following definition.

\begin{definition}[\cite{AK}]
  A permutation $\pi$ is \textbf{$k$-determined} if its corresponding path in the graph $\mathcal{P}_k$ uniquely determines $\pi$.
\end{definition}

Note that the permutations $51432$ and $52431$ from our example above, which have the same path in $\mathcal{P}_3$, are not $2$-bounded.  It turns out that the condition of being $k$-bounded is equivalent to being $(k+1)$-determined, as follows.  

\begin{lemma}\label{lem:k-determined}
  Let $\pi\in S_n$.  The following are equivalent:
  \begin{itemize}
      \item $\pi$ is $(k+1)$-determined
      \item For all $x\in \{1,2,\ldots,n-1\}$, we have $|\pi^{-1}(x)-\pi^{-1}(x+1)|\le k$
      \item $\pi^{-1}$ is $k$-bounded.
  \end{itemize}
Moreover, $\pi^{-1}$ is anchored if and only if $\pi$ is anchored, and so the anchored $k$-bounded permutations of $n$ are in bijection with the anchored $(k+1)$-determined permutations.
\end{lemma}

\begin{proof}
  The equivalence of the first two statements was proven in Theorem 1 of \cite{AK}.  The remaining statements follow immediately from the definitions of $k$-bounded and anchored.
\end{proof}

Lemma \ref{lem:k-determined} allows us to focus instead on the generating function of the anchored $(k+1)$-determined permutations of $n$.  In order to apply the transfer-matrix method, we first recall the notion of a \textbf{prohibited pattern}.

\begin{definition}{\cite{AK}}
   A \textbf{$k$-prohibited pattern} is a pattern of the form $xX(x+1)$ or $(x+1)Xx$ where $X$ is a permutation of $\{1,2,\ldots,|X|+2\}-\{x,x+1\}$ with $|X|\ge k$.
   
   Finally, for any $j$, define $\mathcal{P}_{j,k}$ to be the graph formed by deleting all vertices from $\mathcal{P}_j$ that contain a consecutive pattern that is $k$-prohibited.
\end{definition}

\begin{example}
     The permutation $52431$ exhibits the $2$-prohibited pattern $2431$.
\end{example}

  We say a prohibited pattern $P$ is \textbf{$k$-irreducible} if it has no consecutive patterns of length less than $|P|$ which are $k$-prohibited.  

\begin{lemma}[\cite{AK}]
  Any $k$-irreducible $k$-prohibited pattern has length at most $2k+1$.  Moreover, the $(k+1)$-determined permutations in $S_n$ correspond bijectively to paths of length $n-2k$ in the graph $\mathcal{P}_{2k+1,k}$.
\end{lemma}

We now give a further criterion to determine which paths in $\mathcal{P}_{2k+1,k}$ correspond to the \textit{anchored} $(k+1)$-determined permutations.

\begin{theorem}\label{thm:anchored}
  A $(k+1)$-determined permutation is anchored if and only if its corresponding path in $\mathcal{P}_{2k+1,k}$ starts at a vertex whose pattern starts with $1$ and ends at a vertex whose pattern ends with $2k+1$.
\end{theorem}

\begin{proof}
  The forward implication is clear: if a permutation in $S_n$ is anchored, with first entry $1$ and last entry $n$, then its first consecutive pattern must start with $1$ and its last consecutive pattern (of length $2k+1$) must end with $2k+1$.
  
  Conversely, let $\pi\in S_n$ be $(k+1)$-determined and suppose its first pattern of length $2k+1$ starts with $1$ and its last ends with $2k+1$.  Assume for contradiction that $\pi(1)>1$.  Let $x=\pi(1)$. Then due to the first pattern starting with $1$, the entry $x-1$ does not occur among the first $2k+1$ values of $\pi$.  But then $|\pi^{-1}(x-1)-\pi^{-1}(x)|\ge 2k+1>k$, which contradicts $(k+1)$-determinability by Lemma \ref{lem:k-determined}.  It follows that $\pi(1)=1$, and a similar argument shows that $\pi(n)=n$.  Thus $\pi$ is anchored.
\end{proof}

Theorem \ref{thm:rational} now follows using the transfer-matrix method.  In particular, Theorem 4.7.2 in \cite{Stanley} shows that the generating function for the number of paths of length $n-2k$ between any two fixed vertices in a finite directed graph (in this case $\mathcal{P}_{2k+1,k}$) is rational.  Summing the rational generating functions from all possibilities of valid starting and ending vertices given by Theorem \ref{thm:anchored} gives the desired result.

\section{Structure and enumeration for \texorpdfstring{$k=2$}{k=2} and \texorpdfstring{$k=3$}{k=3}}

We now use Theorem \ref{thm:anchored} to prove Theorem \ref{thm:2-bounded}, the recursion for $k=2$.  

In particular, set $k=2$ and define $A$ to be the adjacency matrix of the graph $\mathcal{P}_{2k+1,k}=P_{5,2}$, that is, the matrix having $A_{ij}=1$ if there is an edge from vertex $i$ to vertex $j$, and $A_{ij}=0$ otherwise.  Define $p_{ij}(n)$ to be the number of paths of length $n$ from vertex $i$ to vertex $j$ in the directed graph $\mathcal{P}_{5,2}$.  Then the transfer-matrix method states that $$G_{ij}(x):=\sum_{n=0}^\infty p_{ij}(n)x^n=\frac{(-1)^{i+j}\det(I-xA;j,i)}{\det(I-xA)}$$ where $\det(I-xA;j,i)$ is the determinant of the minor of the matrix $I-xA$ formed by deleting the $j$th row and $i$th column.  

Note that the sum of the generating functions $G_{ij}(x)$, over all starting vertices $i$ whose permutation starts with $1$ and over all ending vertices $j$ whose permutation ends with $5$, is precisely the generating function $$\sum_{n=0}^\infty R_{n+5}x^{n},$$ since the paths of length $n$ in $\mathcal{P}_{5,2}$ correspond to permutations of length $n+5$.
Using a computer, one easily finds that this summation equals $$(3+x+2x^2)/(1-x-x^3).$$ Since we can explicitly calculate that $R_0=0$, $R_1=R_2=R_3=1$, and $R_4=2$, we have that 
\begin{align*}
    \sum_{n=0}^\infty R_{n}x^n &= x+x^2+x^3+2x^4 + x^5(\sum_{n=0}^\infty R_{n+5}x^n) \\
    &= x+x^2+x^3+2x^4+x^5(3+x+2x^2)/(1-x-x^3) \\
    &= x/(1-x-x^3),
\end{align*}
as desired.

Surprisingly, the same computation for $k=3$ proved to be computationally intractable due to the significantly larger size of the adjacency matrix.  We therefore provide a direct combinatorial proof for the minimal depth recursion for $k=3$.  To do so, and to establish a combinatrial proof for $k=2$, we require some additional notation.

\subsection{Notation}

A \textbf{gap} of a permutation $\pi$ is a difference $\pi(i+1)-\pi(i)$ between two consecutive entries.  We will always write our gaps with a $+$ or $-$ sign in front to indicate the sign, even if the sign is clear, to distinguish gaps from entries.  For instance, we would say that the first gap of the permutation $1,3,2,4$ is $+2$, and the second gap is $-1$.  We sometimes refer to the gaps of a sequence that is not a permutation as well, defined in the same way as consecutive differences between entries. 

A sequence whose gaps are all between $-k$ and $+k$ is said to be \textbf{blocked} or \textbf{stuck} at the end if the last entry $a$ has the property that $a\pm 1, \ldots, a\pm k$ all either occur in the sequence or are less than or equal to $0$.  For instance, if $k=3$, the sequence $1,3,4,6,5,2$ is blocked at $2$; the next possible positive integer that has not been used is $7$, which is more than a gap of $k$ away.

The \textbf{graph} of a permutation of $\{1,\ldots,n\}$ is the plot of all points $(i,\pi(i))$ in the plane.  The \textbf{main diagonal} is the line with equation $y=x$.  Note that a point in the graph of a permutation is on the main diagonal if and only if it is a fixed point of the permutation.

\begin{figure}[H]
\begin{center}
\includegraphics{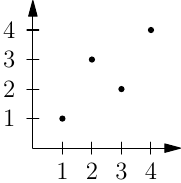}
\end{center}
\caption{The graph of the $2$-bounded permutation $1,3,2,4$.}
\end{figure}

\subsection{Combinatorial proof for \texorpdfstring{$k=2$}{k=2}}\label{sec:2-bounded}

We now give a purely combinatorial proof of Theorem \ref{thm:2-bounded}.  To get a handle on the $2$-bounded anchored permutations, we first prove the following lemma.  It is worth noting that a weaker version of the lemma suffices to prove recursion (\ref{eqn:recursion2}), but the stronger statement explicitly describes the structure of a $2$-bounded permutation.

\begin{lemma}\label{lem:structure2}
  Let $\pi$ be an anchored $2$-bounded permutation of $[n]$.  Then there exists a subset $I\subseteq \{2,\ldots,n-2\}$ such that 
  \begin{enumerate}
    \item Any pair of numbers in $I$ differ by at least three, and 
    \item For all $i\in [n]$, $$\pi(i)=\begin{cases} i+1, & \text{if $i\in I$;} \\ i-1, & \text{if $i-1\in I$;} \\ i, & \text{otherwise.}   \end{cases}$$
  \end{enumerate}
\end{lemma}

In other words, the graph of the permutation can only deviate from the diagonal $x=y$ in consecutive pairs, with an up-step of $2$ and a down-step of $1$, before returning to the diagonal with an up-step of $2$.  (See Figure \ref{fig:2-bounded}.)

\begin{figure}[t]
\begin{center}
\includegraphics{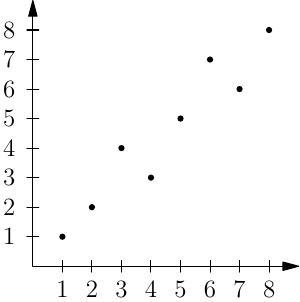}
\end{center}
\caption{\label{fig:2-bounded} The $2$-bounded permutation graphed above, $1,2,4,3,5,7,6,8$, has subset $I=\{3,6\}$ as the set of indices $i$ for which $\pi(i)=i+1$.}
\end{figure}

\begin{proof}
  The lemma is clearly true when $n=1$.  We proceed by strong induction on $n$.  Assume that the lemma holds for all positive integers $n'<n$, and let $\pi$ be a permutation of $[n]$.
  
  If $\pi$ is the identity permutation then $I=\emptyset$ and we are done, so we may assume that $\pi$ is not the identity.  Let $i$ be the smallest index for which $\pi(i)\neq i$.  Note that $i\in\left\{2,\ldots,n-2\right\}$. Then since $\pi(j)=j$ for all $j<i$, the gap from $\pi(i-1)$ to $\pi(i)$ cannot be $-1$, $-2$, or $+1$.  It therefore must be $+2$, and we have 
  $$\pi(i)=\pi(i-1)+2=i-1+2=i+1.$$
  
  Now, the next gap, from $\pi(i)$ to $\pi(i+1)$, can either be $-1$, $+1$, or $+2$.  We claim that it is not $+1$ or $+2$.  If the gap were $+1$, then $i+1$ and $i+2$ both occur, before the value $i$ appears in the permutation.  So for some $j>i+1$, $\pi(j)=i$.  But then the value of $\pi(j+1)$ must be at least $i+3$ (since all other possible values are already used), and this contradicts $2$-boundedness.  Otherwise, if the gap between $\pi(i)$ and $\pi(i+1)$ is $+2$, so that $\pi(i+1)=i+3$, then the only way to reach $i$ in the permutation is via a $-2$ step from $i+2$, and the same argument shows a contradiction.
  
  It follows that the gap at $i$ is $-1$, so $\pi(i+1)=i$.  The only possible value for $\pi(i+2)$ is then $i+2$ (with a $+2$ step from the previous), which is on the diagonal again with all smaller numbers having occurred to the left of it.  The remaining entries form a $2$-bounded, anchored permutation of $\{i+2,i+3,\ldots,n\}$, which has a corresponding subset $I'\subseteq \{i+3,i+4,\ldots,n-2\}$ that satisfies the conditions above by the inductive hypothesis.  Since $i$ is at least $3$ less than any element of $I'$, we see that setting $I=\{i\}\cup I'$ gives a valid subset that corresponds to $\pi$.
\end{proof}

We now can prove Theorem \ref{thm:2-bounded}.

\begin{proof}
It is easily checked that $R_1=R_2=R_3=1$.  Let $n\geq 4$.  Then any anchored $2$-bounded permutation $\pi$ of $[n]$ either starts with $1,2$ or $1,3$.  In the former case, there are $R_{n-1}$ ways of completing the permutation, since any $2$-bounded way of completing it that ends at $n$ is an anchored permutation of $\{2,\ldots,n\}$.  

In the latter case, by Lemma \ref{lem:structure2}, the first four entries of the permutation must be $1,3,2,4$, and then the remaining entries starting from $4$ form $2$-bounded anchored permutation of $\{4,5,\ldots,n\}$.  It follows that there are $R_{n-3}$ possibilities if the permutation starts with $1,3$.  

It follows that $R_n=R_{n-1}+R_{n-3}$.

The generating function now follows from a straightforward calculation.  We have 
\begin{align*}
R(x)-xR(x)-x^3R(x) & = \sum_{n=1}^\infty  R_n x^n - \sum_{n=2}^\infty R_{n-1}x^n - \sum_{n=4}^\infty R_{n-3} x^n \\
 &= x+x^2+x^3-(x^2+x^3)+\sum_{n=4}^\infty (R_n-R_{n-1}-R_{n-3})x^n \\
 &= x + \sum_{n=4}^\infty 0 \cdot x^n \\
 &= x,
\end{align*}
and it follows that $R(x)=x/(1-x-x^3)$.
\end{proof}

\subsection{Combinatorial proof for \texorpdfstring{$k=3$}{k=3}}\label{sec:3-bounded}

As in Theorem \ref{thm:3-bounded}, we define $F_n$ to be the number of $3$-bounded anchored permutations of $[n]$.  In the $2$-bounded case, we saw that there is one possible pattern in which the permutations can veer from the identity, and used that to generate the recursion.  Similarly, in the $3$-bounded case, we will need to single out a certain special sequence that interferes with an otherwise regular pattern that the permutations must follow.

\begin{definition}
  The \textbf{Joker} is the sequence $3,1,4,2,5$.  We say the Joker \textbf{appears} in a $3$-bounded permutation if for some $i$, the $i$th through $(i+4)$th entries of the permutation are $i+2,i,i+3,i+1,i+4$.  
\end{definition}

\begin{figure}[t]
\begin{center}
\includegraphics{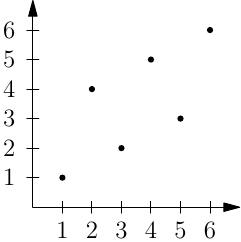}
\end{center}
\caption{\label{fig:Joker} The Joker appears in the above permutation, in its second through sixth entries.}
\end{figure}

Aside from the Joker, the $3$-bounded permutations turn out to follow a predictable pattern in terms of runs of $+3$ and $-3$ steps.  We will use this structure to devise a three-term recurrence for $F_n$.

\begin{definition}
  Define $G_n$ to be the number of $3$-bounded permutations $\pi$ of $\{1,2,\ldots,n\}$ that start with either $\pi(1)=1$ or $\pi(1)=2$ (so they are not necessarily anchored) and end at $\pi(n)=n$. 
\end{definition}

\begin{definition}
 Define $H_n$ to be the number of $3$-bounded permutations $\pi$ of $\{1,2,\ldots,n\}$ that start with $\pi(1)=3$, end with $\pi(n)=n$, and do not start with the Joker as the first five terms.
\end{definition}

We claim that for all $n\geq 6$, the sequences $F_n$, $G_n$, $H_n$ satisfy the following recurrence relations:
\begin{align*}
  F_n &= G_{n-1} + H_{n-1}+F_{n-5}, \\
  G_n &= F_n + G_{n-2} + F_{n-3} + G_{n-4} + H_{n-2}, \\
  H_n &= F_{n-3} + G_{n-3} + F_{n-4} + G_{n-5} + H_{n-3}. 
\end{align*}
  To prove these relations, we first prove the following structure lemma.

\begin{lemma}\label{lem:Joker}
Suppose $\pi$ is a $3$-bounded anchored permutation of $[n]$, and that the first $i$ entries form a $3$-bounded anchored permutation of $[i]$, so that $\pi(1)=1$, $\pi(i)=i$, and the numbers $1,\ldots,i$ comprise the first $i$ entries of the permutation in some order.  If the next step is a $+3$, then one of the following two patterns occurs starting at entry $i$: 
  \begin{enumerate}
    \item\label{option1} The Joker appears as entries $i$ through $i+4$.
    \item\label{option2} There is a positive integer $m$ and a gap $d\in \{\pm 1, \pm 2\}$ such that the sequence of gaps after $i$ is 
      $$+3,+3,\ldots,+3,d,-3,-3,\ldots,-3,\overline{d},+3,+3,\ldots,+3$$ 
    where the first run of $+3$'s has length $m$, the run of $-3$'s has length $m'$ where $m'=m-1$ if $d<0$ and $m'=m$ if $d>0$, the last run of $+3$'s has length $m'$ as well, and 
      $$\overline{d}=
      \begin{cases}
        +1, & \text{if $d=1$ or $d=-2$;} \\
        -1, & \text{if $d=2$ or $d=-1$.}
      \end{cases}$$  
    We call such a pattern a \textbf{cascading $3$-pattern}.
  \end{enumerate}
\end{lemma}

\begin{figure}
\begin{center}
\includegraphics{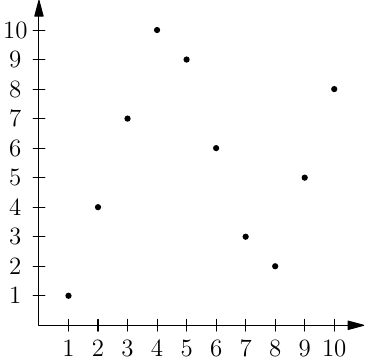}
\end{center}
\caption{\label{fig:cascading3} An example of a cascading $3$-pattern, with $m=3$ and $d=-1$.}
\end{figure}

\begin{proof}
  First, note that since $\pi$ restricts to a permutation on $\{1,\ldots,i\}$, we can assume for simplicity that $i=1$.  Now, suppose the next gap is $+3$, so $\pi(2)=4$.
  
  Let $m$ be the length of the run of consecutive gaps of $+3$ starting from $1$ before a gap $d$ not equal to $+3$ occurs.  Notice that $d$ cannot be $-3$ or else the same entry would occur twice in the permutation, and so $d\in \{\pm 1, \pm 2\}$.  We will prove that one of the two possibilities above hold by induction on $m$.
  
  \textbf{Base Case.} Suppose $m=1$.  We consider several subcases based on the value of $d$.
  
  If $d=-2$, then the first three entries of the sequence are $1,4,2$, and the next entry may be $5$ or $3$.  If the next entry is $5$ and the fifth entry is larger than $5$, then the only way to reach $3$ later in the permutation is by a gap of $-3$ from $6$, in which case we would be stuck at $3$, having used $1$, $2$, $4$, $5$, and $6$ already.  Thus, if $\pi$ starts with $1,4,2,5$ then it must continue $1,4,2,5,3,6$, which is the Joker.  Otherwise, it starts $1,4,2,3$, which is a cascading $3$-pattern for $m=1$ and $d=-2$. 
  
  If $d=-1$, suppose for contradiction that the next gap is positive, so that the first four entries are either $1,4,3,5$ or $1,4,3,6$.  Then $2$ must be reached from a gap of $-3$ from $5$, at which point the permutation is stuck.  Thus the next gap must be $-1$ as well, and the permutation must start $1,4,3,2,5$, which is a cascading $3$-pattern for $m=1$ and $d=-1$.  
  
  If $d=+1$, suppose for contradiction that the next gap is positive, so that the first four entries are either $1,4,5,6$ or $1,4,5,7$ or $1,4,5,8$.  Then to reach $2$ or $3$, there must be a gap of $-3$ from $6$, at which point the permutation is blocked by $4$, $5$, and $6$ and ends at $2$ or $3$, a contradiction.  It follows that the next gap is $-2$ or $-3$, and in fact it must be $-3$ so as to reach the entry $2$ without being blocked.  Thus, the first five entries are $1,4,5,2,3,6$, which is a cascading $3$-pattern with $m=1$ and $d=+1$.
  
  Finally, if $d=+2$, suppose for contradiction that the next gap is positive or $-1$.  Then as in the case above, the permutation becomes blocked once it reaches $2$ or $3$.  So the next gap must be $-3$ and we have $1,4,6,3$ as the first four entries.  We must then have $2$ as the fifth entry, or else the sequence would get blocked at $2$ later, so the first six entries are $1,4,6,3,2,5$, which is a cascading $3$-pattern with $m=1$ and $d=+2$.
  
  \textbf{Induction step.} Suppose $m>1$ and assume the lemma holds for $m'=m-1$.  Then $\pi$ starts with $1,4,7$.  We claim that the entries $2$ and $3$ must be adjacent in $\pi$.  Suppose they are not adjacent.  If $3$ comes first, then the only way to reach $2$ is by a $-3$ gap from $5$ (since $1$ and $4$ are already used) at which point the permutation would be stuck at $2$, a contradiction.  If $2$ comes first, then since $7$ comes after $4$ we must have reached the $2$ using a $-3$ gap from $5$.  But then the only possible entry that can follow the $2$ is $3$, and they are in fact adjacent.
  
  Now, consider the adjacent positions of the $2$ and $3$.  Then the other entry adjacent to $2$ must be $5$, and $6$ must be adjacent to $3$ as well, so the $5$ and $6$ surround the $2$ and $3$.  It follows that if we remove $2$, $3$, and $4$ from the permutation and shift all entries larger than $4$ down by $3$, we obtain a permutation $\pi'$ that starts at $1$ with a $+3$ gap to $4$ (which was the $7$ in $\pi$).  Since the $5$ and $6$ surrounded the $2$ and $3$ in $\pi$, they become $2$ and $3$ and are adjacent in $\pi'$.  All other pairs of adjacent entries in $\pi'$ still have a difference of at most $3$, because they did in $\pi$ and were both translated down by $3$.  Thus, $\pi'$ is a $3$-bounded anchored permutation starting with $m-1$ gaps of $+3$, and by the induction hypothesis it must either start with the Joker or a cascading $3$-pattern.  
  
  Since the $2$ and $3$ are adjacent in $\pi'$ it cannot start with the Joker and so it must be of the second form.  It follows that $\pi$ also starts with a cascading $3$-pattern, formed by inserting one more $+3$ and $-3$ and $+3$ into each of the runs of $3$'s that comprise the gaps of $\pi'$.
\end{proof}

We now prove each of the recurrence relations as their own lemma.

\begin{lemma}\label{lem:Fn-recurrence}
  We have $F_n = G_{n-1} + H_{n-1} + F_{n-5}$.
\end{lemma}

\begin{proof}
  Any $3$-bounded anchored permutation either starts with a gap of $+1$, $+2$, or $+3$.  If it starts with $+1$ or $+2$, together the number of possibilities are equal to the number of $3$-bounded permutations of $\{2,\ldots,n\}$ that start with either $2$ or $3$, which is exactly $G_{n-1}$.
  
  If it starts with $+3$, then by Lemma \ref{lem:Joker} it either starts with the Joker sequence or is a cascading $3$-pattern.  If it starts with the Joker, then $\pi(6)=6$ and the first six entries are a permutation of $[6]$, so the entries after the fifth form a $3$-bounded anchored permutation of $\{6,7,\ldots,n\}$.  There are therefore $F_{n-5}$ possibilities in this case.  Otherwise, the number of possibilities is equal to the number of $3$-bounded permutations of $\{2,\ldots,n\}$ that start with $4$ and end at $n$ but do not start with the Joker, which is exactly $H_{n-1}$.  The recursion follows.
\end{proof}

\begin{lemma}\label{lem:Gn-recurrence}
  We have $G_n=F_n + G_{n-2} + F_{n-3} + G_{n-4} + H_{n-2}$.
\end{lemma}

\begin{proof}
  We now wish to enumerate the $3$-bounded permutations that start at either $1$ or $2$ and end at $n$.  The number starting at $1$ is $F_n$, which is the first term in the recurrence.  
  
  For those starting at $2$, if the next entry is $1$ then the third entry can either be $3$ or $4$.  We now wish to count $3$-bounded permutations of $\{3,\ldots,n\}$ that start at either $3$ or $4$ and end at $n$, which is exactly $G_{n-2}$. 
    
  If the first two entries are $2,3$, then if the next gap is positive it follows that the $1$ can only be reached by a gap of $-3$ from $4$, at which point the permutation is stuck.  It follows that the next gap is negative, and it must be a gap of $-2$.  So the first four entries are $2,3,1,4$, and the remaining entries starting from $4$ form a $3$-bounded anchored permutation of $\{4,\ldots,n\}$. Thus, there are $F_{n-3}$ possibilities in this case.
    
  If the first two entries are $2,4$, then $1$ can either be reached from a gap of $-3$ from $4$, or later from a gap of $-2$ from $3$.  But the latter option becomes stuck at $1$, and so there must be a gap of $-3$ from $4$ to $1$.  It follows that the permutation starts $2,4,1,3$ and then continues with a $3$-bounded permutation of $\{5,\ldots,n\}$ that starts at either $5$ or $6$.  There are therefore $G_{n-4}$ such possibilities.
    
  Finally, if the first two entries are $2,5$, then the $1$ must occur at some point in $\pi$ and must be surrounded by $3$ and $4$.  If we remove the $1$, then, we obtain a $3$-bounded permutation of $\{2,\ldots,n\}$ starting at $2$ and with a starting gap of $+3$, with the $3$ and $4$ adjacent.  By Lemma \ref{lem:Joker}, the $3$ and $4$ will always be adjacent in such a permutation with a starting gap of $+3$ unless it starts with the Joker pattern, and so, removing the $1$ and the $2$, we see that there are exactly $H_{n-2}$ possibilities in this case.
\end{proof}

Notice that the final step in the above proof was analogous to the final step of the proof of Lemma \ref{lem:Fn-recurrence}.  Deleting the $1$ from the permutation resulted in the $H_{n-1}$ term in the $F_n$ recurrence, just as deleting the $1$ and the $2$ from the permutation resulted in the $H_{n-2}$ term in the $G_n$ recurrence.  We will use this trick once more below, deleting the $1$, $2$, and $3$, resulting in a $H_{n-3}$ term in the $H_n$ recurrence.

\begin{lemma}\label{lem:Hn-recurrence}
  We have $H_n=F_{n-3} + G_{n-3} + F_{n-4} + G_{n-5} + H_{n-3}$.
\end{lemma}

\begin{proof}
  We wish to enumerate the $3$-bounded permutations that start at $3$ and end at $n$ but do not start with the Joker sequence $3,1,4,2,5$.  The second entry can either be $1$, $2$, $4$, $5$, or $6$.
   
  Notice that if we add a $0$ to the front of the permutation, we will get a $3$-bounded anchored permutation of $\{0,\ldots,n\}$ that starts with a gap of $+3$.  By Lemma \ref{lem:Joker}, since the permutation does not start with the Joker, it must start with a cascading $3$-pattern.  
   
  Thus, if the first gap after the $3$ is not $+3$, then $d$ is determined and the $3$-pattern is determined as well.  In particular, if the first two entries are $3,1$ then the permutation must start with $3,1,2$, and so the entries after the third form a $3$-bounded permutation of $\{4,\ldots,n\}$ that starts at either $4$ or $5$ and ends at $n$.  There are exactly $G_{n-3}$ such entries in this case.
   
  If the first two entries are $3,2$ then since the start is a cascading $3$-pattern, the first four entries are $3,2,1,4$.  The entries starting at $4$ form a $3$-bounded permutation of $\{4,\ldots,n\}$ starting at $4$ and ending at $n$, giving us $F_{n-3}$ more possibilities.
   
  If the first two entries are $3,4$, then by the cascading $3$-pattern the first five entries are $3,4,1,2,5$.  The entries starting at $5$ form a $3$-bounded permutation of $\{5,\ldots,n\}$ starting at $5$ and ending at $n$, giving us $F_{n-4}$ more possibilities.
   
  If the first two entries are $3,5$, the cascading $3$-pattern tells us that the first five entries are $3,5,2,1,4$, with the next entry either $6$ or $7$.  The entries starting after the fifth form a $3$-bounded permutation of $\{6,\ldots,n\}$ starting at either $6$ or $7$ and ending at $n$, giving us $G_{n-5}$ more possibilities.
   
  Finally, if the first two entries are $3,6$, then since it is a cascading $3$-pattern the $1$ and $2$ must be adjacent in $\pi$.  Removing the $1$, $2$, and $3$ then gives a $3$-bounded permutation of $\{4,\ldots,n\}$ that starts at $6$ and ends at $n$ but avoids the Joker.  There are $H_{n-3}$ such possibilities, and the proof is complete.
\end{proof}

We can now eliminate $H_n$ from these recurrences to form a two-term recurrence.  Putting $n-1$ in the recurrence for $G_n$, we have $G_{n-1}=F_{n-1}+G_{n-3} + F_{n-4} + G_{n-5} + H_{n-3}$, which nearly matches the recurrence for $H_n$.  From this we conclude $H_n=F_{n-3}+G_{n-1}-F_{n-1}$.  We can now substitute for the $H$ terms in the $F$ and $G$ recurrences to obtain the following relationships:
  \begin{align}
    F_n &=G_{n-1}+F_{n-4}+G_{n-2}-F_{n-2}+F_{n-5}, \label{eqn:recursionF}\\
    G_n &=F_n+G_{n-2}+G_{n-3}+G_{n-4}+F_{n-5}.  \label{eqn:recursionG}
  \end{align}
Notice that our proofs above actually show that these recursions hold for all $n$, even $n\leq 5$, where we set $F_j=G_j=0$ for any $j\leq 0$.  Thus, we can unwind the recursions to find the first few values of $F_n$ and $G_n$, as follows.

\begin{center}
 \begin{tabular}{c|cccccccc}
    $n$ & $1$ & $2$ & $3$ & $4$ & $5$ & $6$ & $7$ & $8$  \\ \hline
     $F_n$ & $1$ & $1$ & $1$ & $2$ & $6$ & $14$ & $28$ & $56$  \\
     $G_n$ & $1$ & $1$ & $2$ & $4$ & $10$ & $22$ & $45$ & $93$ \\
  \end{tabular}
\end{center}

We now have the tools to prove Theorem \ref{thm:3-bounded}.

\begin{proof}
  We first find the generating function for $\{F_n\}$, and use this to find the single-term recurrence for the sequence.

  Let $F(x)=\sum_{n=1}^\infty F_n x^n$ and $G(x)=\sum_{n=1}^\infty G_n x^n$.  Then we have 
    \begin{align*}
      F(x) & = x+x^2+x^3+2x^4+6x^5 + \sum_{n=6}^\infty F_n x^n \\
      x^2F(x)&=\hspace{1.6cm}x^3+\phantom{2}x^4+\phantom{6}x^5 + \sum_{n=6}^\infty F_{n-2} x^n \\
      x^4F(x)&=\phantom{x+x^2+x^3+2x^4+6}x^5+  \sum_{n=6}^\infty F_{n-4} x^n \\
      x^5F(x) &=\hspace{4.8cm} \sum_{n=6}^\infty F_{n-5} x^n,
    \end{align*}
  and 
    \begin{align*}
      G(x) & = x+x^2+2x^3+4x^4+10x^5 + \sum_{n=6}^\infty G_n x^n \\
      xG(x)& = \hspace{0.75cm}x^2+\phantom{2}x^3+2x^4+\phantom{1}4x^5 + \sum_{n=6}^\infty G_{n-1} x^n \\
      x^2G(x)&=\phantom{x+x^2+2} x^3+\phantom{2}x^4+\phantom{1}2x^5 + \sum_{n=6}^\infty G_{n-2} x^n \\
      x^3G(x)&=\hspace{3cm}x^4+\phantom{10}x^5 + \sum_{n=6}^\infty G_{n-3} x^n \\
      x^4G(x)&=\hspace{4.3cm}x^5+  \sum_{n=6}^\infty G_{n-4} x^n. \\
    \end{align*}

  We can now utilize the recursions (\ref{eqn:recursionF}) and (\ref{eqn:recursionG}) to make the infinite summations cancel and keep track of the smaller terms, obtaining the following two equations:
    \begin{align*}
      F(x)-xG(x)-x^2G(x)+x^2F(x)-x^4F(x)-x^5F(x)&=x,\\
      G(x)-F(x)-x^2G(x)-x^3G(x)-x^4G(x)-x^5F(x)&=0.
    \end{align*}
  Solving these two equations for $F(x)$ and $G(x)$ gives us that $$F(x)=\frac{x-x^2-x^4}{1-2x+x^2-2x^3-x^4-x^5+x^7+x^8}.$$  Finally, we can multiply both sides of the above relation by the denominator of the fraction, and we find that for $n\geq 8$, $F_n$ satisfies the recursion 
    $$F_n=2F_{n-1}-F_{n-2}+2F_{n-3}+F_{n-4}+F_{n-5}-F_{n-7}-F_{n-8},$$ 
  as desired.
\end{proof}

\section{Asymptotic bounds}\label{sec:asymptotics}

We now establish asymptotic bounds for $F_n^{(k)}$.  To do so, recall that a directed graph is \textbf{strongly connected} if there is a directed path from any vertex to any other vertex.   We will make use of known results on the spectra of the adjacency matrix of strongly connected directed graphs.  As $\mathcal{P}_{2k+1,k}$ itself is not strongly connected, we begin by restricting to a certain strongly connected component of $\mathcal{P}_{2k+1,k}$ before applying the transfer-matrix method.

\begin{definition}
  Define $U_{2k+1,k}$ (resp.\ $W_{2k+1,k}$) to be the set of vertices in $\mathcal{P}_{2k+1,k}$ whose pattern starts with 1 (resp.\ ends with $2k+1$).
\end{definition}

By Theorem \ref{thm:anchored}, the sets $U_{2k+1,k}$ and $W_{2k+1,k}$ are the sets of possible starting and ending vertices for a path in $\mathcal{P}_{2k+1,k}$ that determines an anchored $(k+1)$-determined permutation. 

\begin{definition}
  Let $V'_{2k+1,k}$ denote the set of vertices $v$ in $\mathcal{P}_{2k+1,k}$ for which there exists a directed path containing $v$ that starts at some vertex $u\in U_{2k+1,k}$ and ends at some vertex $w\in W_{2k+1,k}$.   Then we define $\mathcal{P}'_{2k+1,k}$ to be the vertex-induced subgraph of $\mathcal{P}_{2k+1,k}$ on the vertices $V'_{2k+1,k}$.
\end{definition}

\begin{lemma}\label{lem:strongly-connected}
  The graph $\mathcal{P}'_{2k+1,k}$ is strongly connected, and moreover, it is the strongly connected component in $\mathcal{P}_{2k+1,k}$ of the vertex labeled by the identity permutation $\mathrm{id}_{[2k+1]}\in S_{2k+1}$.
\end{lemma}

\begin{proof}
  Let $v\in V'_{2k+1,k}$ be a vertex in $\mathcal{P}'_{2k+1,k}$.  Then there exists $u\in U_{2k+1,k}$ and $w\in W_{2k+1,k}$ for which there is a directed path $u\to v\to w$ in $\mathcal{P}_{2k+1,k}$.  
  
  Let $e$ denote the identity vertex $\mathrm{id}_{[2k+1]}$.  We first show that there are directed paths $e\to v$ and $v\to e$ in $\mathcal{P}_{2k+1,k}$.  
  To show that there is a directed path $e\to v$, it suffices to construct a path $e\to u$, since there is a path $u\to v$.  By the definition of $U_{2k+1,k}$, we have that $u=u_1\cdots u_{2k+1}$ is a permutation pattern having $u_1=1$, and the inverse of $u$ is $k$-bounded since it is a non-prohibited pattern.  We can therefore define the permutation $p=p_1\cdots p_{4k+2}$ by $$p_i=\begin{cases} i, & i\le 2k+1; \\ u_{i-2k-1}+2k+1, & i> 2k+1.\end{cases}$$
  In other words, we increase each of the entries of $u$ by $2k+1$ and append the result to $e$, to form the permutation $p$ of length $4k+2$.  
  
  Note that $2k+1$ and $2k+2$ are adjacent in $p$, since $u_1=1$ by assumption.  Furthermore, $e$ and $u$ both have $k$-bounded inverses.  Thus, no pair of consecutive numbers $i,i+1$ occur more than a distance of $k$ apart in $p$.  Hence $p^{-1}$ is $k$-bounded as well, and so $p$ is $(k+1)$-determined by Lemma \ref{lem:k-determined}.  It follows that there is a path in $\mathcal{P}_{2k+1,k}$ giving the consecutive patterns of $p$, and this path starts at $e$ and ends at $u$, as desired.
  
  To show there is a path from $v$ to $e$, it similarly suffices to find a path from $w$ to $e$.  Since $w$ ends with $2k+1$, we can add $2k+1$ to each entry of $e$ and append it to the end of $w$ to form a longer permutation, and a similar argument as above gives the desired path from $w$ to $e$ in $\mathcal{P}_{2k+1,k}$.
  
  Note that the directed paths $v\to e$ and $e\to v$ lie entirely in $\mathcal{P}'_{2k+1,k}$ (since $e\in U_{2k+1,k}$ and $e\in W_{2k+1,k}$). Thus, if $v$ and $t$ are any two vertices in $\mathcal{P}'_{2k+1,k}$, they are connected by a directed path $v\to e \to t$.  Hence $\mathcal{P}'_{2k+1,k}$ is strongly connected.
  
  Finally, given any vertex $v$ in the strongly connected component of $e$ in $\mathcal{P}_{2k+1,k}$, some path $e\to v\to e$ exists by the definition of strongly connected.  Hence $v\in \mathcal{P}'_{2k+1,k}$ since $e\in U_{2k+1,k}\cap W_{2k+1,k}$.  It follows that $\mathcal{P}'_{2k+1,k}$ is indeed equal to the strongly connected component of $e$.
\end{proof}

Now, let $A'$ be the adjacency matrix of $\mathcal{P'}_{2k+1,k}$.   By the transfer-matrix method and Theorem \ref{thm:anchored}, the denominator of one rational expression for the generating function $$F^{(k)}(x):=\sum F^{(k)}_n x^n$$ is given by $\det(I-xA')$.  The reversed polynomial of $\det(I-xA')$ is simply the characteristic polynomial $\det(A'-xI)$ of $A'$, whose roots are given by the eigenvalues of $A'$.  

Using the well-known Perron-Frobenius theorem, we will show that $A'$ has a unique eigenvalue $r$ with largest absolute value, and that $r$ is a positive real number with multiplicity $1$.  Then, by considering the partial fraction decomposition of the rational generating function (after factoring the denominator completely over $\mathbb{C}$), we will see that $F^{(k)}_n$ is bounded above by $c\cdot r^n$ for some constant $c>0$ for sufficiently large $n$.

The ``nonnegative'' version of the Perron-Frobenius theorem can be stated as follows.  Recall that the \textbf{spectral radius} of a matrix is the largest absolute value of any complex eigenvalue.  The following theorem summarizes several of the results described in Section 8.4 of \cite{HornJohnson}.

\begin{theorem}[Perron-Frobenius]
  Let $A$ be a matrix having nonnegative real entries, such that for any $i$ and $j$ there is some power $A^n$ for which $A^n_{ij}>0$.  Then there is a unique positive real eigenvalue, of multiplicity one, equal to the spectral radius of $A$.  
  
  Moreover, the number of eigenvalues of $A$ equal to the spectral radius is equal to the \textbf{period} of $A$, defined as greatest common divisor of all exponents $n$ for which some diagonal entry $(A^n)_{i,i}$ is nonzero.
\end{theorem}

Note that the condition on $A$ in the theorem above, for an adjacency matrix of a directed graph, is precisely equivalent to the graph being strongly connected.  

Lemma \ref{lem:strongly-connected} therefore allows us to apply the Perron-Frobenius theorem to the adjacency matrix $A'$ of the strongly connected graph $\mathcal{P}'_{2k+1,k}$.  In particular, note that if $i$ is the index of the identity vertex $e$, we see that $A_{i,i}'=1$ since there is a directed edge from $e$ to itself.  Hence the period of $A'$ is $1$, and so there is a unique maximal eigenvalue $r$, and this eigenvalue is positive and real of multiplicity one.

 We now wish to find an upper bound for $r$.  Frobenius obtained classical bounds for the spectral radius of an adjacency matrix of a strongly connected digraph.  In particular, it must be less than or equal to the maximum outdegree of the vertices.  (See Theorem 2.1 in the survey paper \cite{Brualdi}.)

Therefore, to prove Theorem \ref{thm:asymptotics}, it only remains to show that $\mathcal{P'}_{2k+1,k}$ has a maximum outdegree of $k$. 

\begin{lemma}\label{lem:outdegree}
  The largest outdegree of any vertex in $\mathcal{P'}_{2k+1,k}$ is $k$.
\end{lemma}

\begin{proof}
Let $v=v_1\cdots v_{2k+1}$ be a vertex in $\mathcal{P'}_{2k+1,k}$.  Let $$a_1a_2\cdots a_kb_1b_2\cdots b_k $$ be the permutation pattern of $v_2v_3\cdots v_{2k+1}$.  

In the larger graph $\mathcal{P}_{2k+1}$, the vertex $v$  has outdegree of exactly $2k+1$ since any number $m\in \{1,2,\ldots ,2k+1\}$ can be appended to the pattern $a_1\cdots a_kb_1\cdots b_k$, incrementing all entries $a_i\ge m$ or $b_j\ge m$ by $1$, in order to form a vertex that $v$ points to.  However, at least $k+1$ of the values of $m$ lead to $k$-prohibited permutation patterns.  Specifically, if we insert any of $a_1,a_2,\ldots ,a_k,$ or $1+\text{max}\left\lbrace a_1,a_2,\ldots, a_k\right\rbrace$ as the last entry, we will create a $k$-prohibited pattern, and the above listed numbers are distinct.  Thus, the maximum outdegree in the smaller graph $\mathcal{P'}_{2k+1,k}$ is no greater than $(2k+1)-(k+1)=k$.  

Finally, we show that the identity pattern $e=1,2,3,\ldots,2k+1$ has outdegree exactly $k$.  Indeed, consider the permutations of length $2k+2$ of the form $$1,2,3,\ldots,i-1,i+1,\ldots,2k+2,i.$$  Such a permutation is $(k+1)$-determined if and only if $i\in \{k+3,\ldots,2k+2\}$, and each of these $k$ permutations gives a directed edge $e\to v$ for some pattern $v$.  Thus $e$ has outdegree $k$ in $\mathcal{P}_{2k+1,k}$.  To see that each of these vertices $v$ is also in $\mathcal{P}'_{2k+1,k}$, note that we may extend the permutation above to $$1,2,3,\ldots,i-1,i+1,\ldots,2k+2,i,2k+3,2k+4,\ldots,4k+3,$$ which is $(k+1)$-determined and therefore gives a path from $v$ back to the identity in $\mathcal{P}_{2k+1,k}$.  Hence $v\in \mathcal{P}'_{2k+1,k}$ as desired.
\end{proof}

Thus, we have $r\leq k$, and the proof of Theorem \ref{thm:asymptotics} is complete.

\begin{remark}
In the cases $k=2$ and $k=3$, we can directly factor the (degree-reversed) denominators of the rational generating functions stated in Theorems \ref{thm:2-bounded} and \ref{thm:3-bounded} to find that the largest roots are approximately $1.466$ and $2.114$ respectively, when rounded up to three decimal places.  Thus $R_n=O(1.466^n)$ and $F_n=O(2.114^n)$, which agrees with the upper bounds of $O(2^n)$ and $O(3^n)$ given by Theorem \ref{thm:asymptotics}.
\end{remark}

\begin{remark}
  A lower bound for $F_n^{(k)}$, similar to that in \cite{AK} for the non-anchored case, may be obtained as follows.  Consider a permutation $\pi:[n]\to [n]$ with the following properties:
  \begin{enumerate}
      \item The map $\pi$ fixes the elements $1,k+2,2k+3,\ldots$, namely those of the form $(k+1)a+1$ for some nonnegative integer $a$.
      \item It restricts to a permutation on each of the consecutive blocks of elements $$\{2,3,4,\ldots,k+1\},\{k+3,k+4,\ldots,2k+2\},\ldots ,\{(k+1)(b-1)+1,\ldots,(k+1)b\} $$  where $(k+1)b+1$ is the largest fixed point of the form described above.
      \item The remaining elements $(k+1)b+1,(k+1)b+2,\ldots,n$ are also fixed by $\pi$.
   \end{enumerate}
   Then $\pi^{-1}$ is $k$-bounded.  Since the second property above gives $k!$ options for the permutations on each block, and there are $\lfloor \frac{n-1}{k+1} \rfloor$ such blocks, we have that there are at least $$(k!)^{\lfloor (n-1)/(k+1) \rfloor}$$ $k$-bounded permutations of length $n$.
\end{remark}
\section{Acknowledgments}

Thanks to the anonymous referee of a previous version of this paper for providing constructive feedback.  Thanks also to Michael Albert for sharing some additional thoughts pertaining to this topic.

\end{document}